\documentclass[11pt]{article}
\usepackage{graphicx} 
\usepackage{amsmath,amsfonts,amsthm}
\usepackage{amssymb}
\usepackage{xcolor}
\usepackage{dsfont}
\usepackage{bbm}
\usepackage{bm}
\usepackage{enumitem}
\usepackage{tikz}

\newcommand*{\empirical}{\widehat{M}}

\newcommand*{\entropy}{H}

\newcommand*{\N}{\mathbb{N}}
\newcommand*{\Prob}{\mathbb{P}}
\newcommand*{\GWT}{\mathbb{T}}
\newcommand*{\E}{\mathbb{E}}
\newcommand*{\e}{\mathsf{e}}
\newcommand*{\typ}{\mathsf{Typ}}
\newcommand*{\typind}{\mathcal{T}yp}
\newcommand*{\odd}{\mathsf{Odd}}
\newcommand*{\oddind}{\mathcal{O}dd}
\newcommand*{\avez}{\textsf{h}(\mu)}

\newcommand*{\BRW}{\textsf{BRW}}
\newcommand*{\indempirical}{\widehat{\mathcal{M}}}
\newcommand*{\ent}{\textsf{Ent}}

\newtheorem{theorem}{Theorem}
\newtheorem{lemma}{Lemma}

\newtheorem{remark}{Remark}

\usepackage[left=1in, right=1in, bottom = 1in, top = 1in]{geometry}
\definecolor{caribbeangreen}{rgb}{0.0, 0.8, 0.6}
\definecolor{brass}{rgb}{0.71, 0.65, 0.26}
\definecolor{amethyst}{rgb}{0.6, 0.4, 0.8}

\newcommand{\rk}[1]{\textcolor{blue}{#1}} 

\title{Phase transition for the asymptotic entropy of branching random walks on groups}
\author{Jeremie Brieussel, Robin Kaiser, Martin Klötzer, Ecaterina Sava-Huss}
\date{\today}

\begin{document}

\maketitle

\begin{abstract}
We consider supercritical branching random walks (BRW shortly) on countable 
groups $G$ and we prove that the asymptotic entropy of the empirical distributions of the BRW has a phase transition at $\rho_* = \e^{\avez}$, where $\avez$ is the asymptotic entropy of the underlying random walk on $G$ with step distribution $\mu$.
Below this value $\rho_*$, the asymptotic empirical entropy of BRW equals the logarithm of the exponential growth rate of the population. Above this value, it is constantly equal to the asymptotic entropy of the underlying random walk. In particular, this answers questions from Kaimanovich-Woess \cite[Section 6.3]{WoessKaim} about the existence and the behavior of the asymptotic entropy.
\end{abstract}

\section{Introduction}

Given a countable group $G$, an offspring distribution $\pi \in \textsf{Prob}(\mathbb{N}_0)$ and a probability measure $\mu \in \textsf{Prob}(G)$, a branching random walk $\mathrm{BRW}(G,\pi,\mu)$ on $G$ couples the dynamics of a Galton-Watson process to the dynamics of a random walk on $G$ with step distribution $\mu$. In each discrete time step,  every particle alive independently generates a $\pi$-distributed number of offspring, and each of these offspring subsequently takes one step according to $\mu$.
It is well known that when spatial movement on $G$ is ignored (that is, when looking only at the Galton–Watson process with offspring distribution $\pi$), the model exhibits an extinction–survival phase transition governed by $\rho=\mathbb{E}\pi$, the expected number of offspring of a single particle. This transition occurs at the critical value $\rho=1$; when  $\rho\leq 1$, the process dies out  almost surely, and if $\rho>1$, then it survives with positive probability. 

On countable 
groups $G$, branching random walks exhibit not only the classical extinction–survival phase transition, but also a second one: the transition from weak to strong survival. We say that $\mathrm{BRW}(G,\pi,\mu)$ survives weakly if the population survives almost surely, but every finite subset of $G$ will eventually be vacated and never visited again. This second threshold occurs at  $\rho = 1/r$, 
where $r\in(0,1]$ is the spectral radius of the underlying random walk $(Y_n)_{n\in \N}$ with step distribution $\mu$; see Benjamini-Peres \cite{recurrence-brw-benjamini-peres-1994} and Müller \cite{mueller-brw} for details. By Kesten \cite{Kesten_amenability}, 
for symmetric and non-degenerate measures $\mu$, the spectral radius $r$ is equal to one if and only if $G$ is amenable. In particular the phase transitions extinction-survival and weak-strong survival are distinct if and only if $G$ is nonamenable.  

The phase transition at $1/r$ is also reflected in the boundary behaviour of  branching random walks. Hueter-Lalley \cite{hueter-lalley-brw-trees-2000} considered BRW on homogeneous trees $\mathbb{T}$ and proved that if $\rho\leq 1/r$,  then the Hausdorff dimension of the limit set $\Lambda$ of the branching random walk with respect to a natural metric on the boundary $\partial \mathbb{T}$ of $\mathbb{T}$ is almost surely constant and obtained a precise formula for it. Their result has been extended to BRW on free products of finitely generated groups in Candellero-Gilch-Müller \cite{BRWOnFreeProducts}. On hyperbolic groups in Sidoravicius-Wang-Xiang \cite{LimSetHyperbolic} it is shown that the Hausdorff dimension of the limit set $\Lambda$ of BRW is less or equal than half of the dimension of the hyperbolic boundary of the underlying group $G$. In the same regime $\rho\in(1,1/r\big]$, it has been proven by Dussaule-Wang-Yang \cite{dussaule-wang-yang-brw-hyper} that the trace of the BRW on relatively hyperbolic groups equals the growth rate of the Green function of the underlying random walk. 

Considering the population statistics of branching random walks, several aspects have been investigated in the past. Stam \cite{stam} proved that the number of particles at distance of $\sqrt{n}$ around the speed of the underlying random walk is described by a normal distribution. This result has recently been extended to arbitrary transitive graphs in \cite{Kloetzer-Kaiser-Sava-Huss-EJP26}. Another aspect that has been studied intensively, is the maximal distance travelled by the ensemble of branching particles; see \cite{logcorr2, maxconv, berestycki2026biasedbranchingrandomwalks}. 
For branching random walks on general state spaces which are endowed with a compactification (e.g.\ end compactification, Martin compactification), it has been recently shown in Candellero-Hutchcroft \cite{Candellero-Hutchcroft-BRW-2023} and Woess-Kaimanovich \cite{WoessKaim} that the empirical distribution of the population converges almost surely (in the weak sense) to a random measure on the boundary of the respective compactification.
The empirical distribution of the branching random walk $\mathrm{BRW}(G,\pi,\mu)$ in generation $n$ is defined as the random probability measure
$$\empirical_n:G\to [0,1]:x\mapsto \frac{\#\{\text{Individuals at } x \text{ in generation } n\}}{Z_n},$$
where $Z_n$ is the total number of particles alive in generation $n$. 
Limit results for the empirical distributions and for the maximum and minimum displacement in branching random walks on groups have been recently investigated in Klötzer-Kaiser-Sava-Huss \cite{Kloetzer-Kaiser-Sava-Huss-EJP26} and Klötzer-Kolesko-Kaiser-Sava-Huss \cite{Kloetzer-Kolesko-Kaiser-Sava-Huss-MaxMin-2026}. There is a huge amount of literature on several aspects of branching random walks and their limit behaviour, but 
most of them deal with the case when the underlying state space is $\mathbb{Z}^d$. When the underlying state space has itself a rich geometry at infinity, not much is known and it deserves future investigation. Recently, Geldbach \cite{Geldbach-2026-BBM} investigated the Hausdorff dimension of the empirical limit measure for branching Brownian motion on hyperbolic spaces.

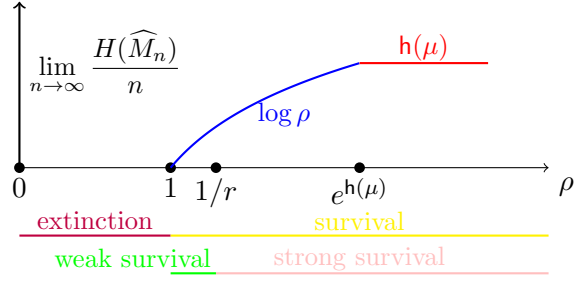
\begin{figure}
\centering
\begin{tikzpicture}
\draw[->] (0,0) -- (7,0);

\draw[->, thick] (0,0) -- (0,2.2);

\node[right] at (0,1.4) {\small $\displaystyle \lim_{n\to\infty}\frac{H(\widehat{M}_n)}{n}$};

\fill (0,0) circle (2pt);
\fill (2,0) circle (2pt);
\fill (2.6,0) circle (2pt);
\fill (4.5,0) circle (2pt);

\node[below] at (0,0) {$0$};
\node[below] at (2,0) {$1$};
\node[below] at (2.6,0) {$1/r$};
\node[below] at (4.5,0) {$e^{\avez}$};

\node[below right] at (7,0) {$\rho$};

\draw[thick, blue, domain=0:1, samples=100]
plot ({2+2.5*\x},{ln(1+3*\x)});

\node[blue] at (3.5,0.7) {\small $\log\rho$};

\draw[thick, red] (4.5,1.386) -- (6.2,1.386);

\node[red] at (5.35,1.6) {\small $\avez$};

\draw[thick, purple] (0,-0.9) -- (2,-0.9);
\draw[thick, yellow] (2,-0.9) -- (7,-0.9);

\node[purple] at (1,-0.7) {\small extinction};
\node[yellow] at (4.5,-0.7) {\small survival};

\draw[thick, green] (2,-1.4) -- (2.6,-1.4);
\draw[thick, pink] (2.6,-1.4) -- (7,-1.4);

\node[green] at (1.5,-1.2) {\small weak survival};
\node[pink] at (4.5,-1.2) {\small strong survival};
\end{tikzpicture}
\caption{The three phase transitions in a branching random walk}\label{fig:phase-transition}
\end{figure}

For a random walk with step distribution $\mu$ on a group $G$, an important statistic is its Avez asymptotic entropy, defined as $\avez := \lim_n H(\mu^{*n})/n$, where $H(\mu^{*n})$ is the Shannon entropy of $\mu^{*n}$~\cite{Avez}. Its relevance is firstly due to the harmonic analysis fact that it vanishes if and only if every $\mu$-harmonic function is constant; that is, the random walk has a trivial Poisson boundary~\cite{Der80,KV83}. From the point of view of measured dynamics, the asymptotic entropy equals the Furstenberg entropy of the Poisson boundary and is the maximal value of Furstenberg entropy over all $\mu$-stationary $G$-actions~\cite{KV83}. The Furstenberg entropy is an invariant of an action measuring its lack of invariance~\cite{KV83,NZ1}. Let us also mention a geometric aspect that the asymptotic entropy equals the rate of escape of the random walk with respect to the Green distance \cite{BP94,BHM08}. Regarding the rate of escape $\ell$ of the random walks with respect to a word metric, the asymptotic entropy satisfies the fundamental Guivarch'inequality $\avez \le \ell v$, where $v$ is the exponential volume growth~\cite{Guivarch80}. For hyperbolic groups, the equality case corresponds to the situation where the harmonic measure on the visual boundary is in the measure class of the Patterson-Sullivan measure~\cite{GMM18}. 

The Avez entropy can also be compared to the spectral radius. If the random walk $(Y_n)_n$ is symmetric, it is also well known that $-\log r \leq \avez$ (see~\cite{Avez}). A sharper lower bound $\avez \ge 2\sqrt{1-r}\ \mathrm{arctanh}(\sqrt{1-r})$ was given in~\cite[Theorem 1.2]{GMM} (note the different definition of $r$ in this reference). This sharper lower bound implies in particular that the Avez inequality is strict $-\log r < \avez$, unless $\avez=0$. The case when $\avez=0$ is often referred to as a Liouville random walk in literature.

The rate of escape, for a word metric, of $\mathrm{BRW}(G,\pi,\mu)$ and its relation with the rate of escape of the underlying random walk $(G,\mu)$ has been recently investigated in Klötzer-Kaiser-Sava-Huss~\cite{Kloetzer-Kaiser-Sava-Huss-EJP26}. In the current work we investigate the asymptotic entropy of the empirical distributions of $\mathrm{BRW}(G,\pi,\mu)$ and we prove that there is a third phase transition going from a small-$\rho$ regime in which we have asymptotically maximal entropy to a large-$\rho$ regime where we have asymptotically an entropy which stays constant as $\rho$ increases. This behaviour and the corresponding phase transitions in the context of branching random walks are graphically illustrated in Figure~\ref{fig:phase-transition}.
The critical value for this phase transition occurs at $\rho_*=\e^{\avez}$. Note that for nonamenable groups, the quantity $\e^{\avez}$ is always bigger than $1$, but the inequality $\e^{\avez}>1$ does not imply nonamenability, as the random walk with drift on the lamplighter group $\mathbb{Z}_2 \wr \mathbb{Z}$ provides a counterexample~\cite{KV83}.

The focus of the current work is on the (random) Shannon entropy $\entropy(\empirical_n)$ of the empirical distributions $\empirical_n$ of the branching random walk $\mathrm{BRW}(G,\pi,\mu)$, which is defined as
\begin{equation}\tag{Ent-BRW}\label{eq:shannon-emp-distr}
\entropy(\empirical_n)=-\sum_{x\in G}\empirical_n(x)\log\empirical_n(x).  
\end{equation}
Our work is motivated by the questions posed in Kaimanovich-Woess \cite[Section 6.3.]{WoessKaim}. Does the asymptotic entropy $\lim_n H(\empirical_n)/n$ exist almost surely, and what is its exact value given its existence? To the best of our knowledge, this is the first paper that addresses the notion of asymptotic entropy for the empirical distributions of branching random walks. Our main result answers the question of existence, and determines the exact value of the asymptotic entropy, based on the mean offspring number $\rho$ and the  step distribution $\mu$ on $G$.

\begin{theorem}[Branching random walks]\label{thm: main 1}
Let $G$ be a countable group and let $\BRW(G,\pi,\mu)$ be a branching random walk with step distribution $\mu$ on $G$ and offspring distribution $\pi$ with mean $\E[\pi]=\rho>1$. If $\mu$ has finite entropy, $\pi(0)=0$, and $\pi$ has a finite second moment, then the following holds for the entropy $H(\empirical_n)$ of the empirical distribution $\empirical_n$ of $\BRW(G,\pi,\mu)$\rk{:}
\begin{enumerate}[label = (\roman*)]
\setlength\itemsep{0em}
    \item \textbf{Small-$\rho$ regime.} If $\rho\leq \e^{\avez}$, then almost surely
    $$\lim_{n\to\infty}\frac{H(\empirical_n)}{n}=\log\rho.$$
    \item \textbf{Large-$\rho$ regime.} If $\rho > \e^{\avez}$, then almost surely
    $$\lim_{n\to\infty}\frac{H(\empirical_n)}{n} = \avez,$$
    where $\avez$ is the asymptotic entropy of the random walk $(G,\mu)$.
\end{enumerate}
\end{theorem}
As we see from Theorem \ref{thm: main 1}, the asymptotic entropy undergoes a phase transition that can be described as follows: if the mean offspring number $\rho$ is too small (specifically, if $\log\rho\leq \avez$), the amount of particles is not sufficient to properly probe the underlying step distribution $\mu$; the asymptotic entropy is capped at $\log\rho$, and this describes the maximum information the branching particles can obtain. If $\rho$ becomes large enough (specifically, if $\log\rho>\avez$), the cluster of particles can properly recover enough information about $\mu$, and the asymptotic entropy evaluates to $\avez$. 

Our result implies that heuristically, there are three phases for surviving BRWs on a group, all distinct in the case of a nonamenable group. Firstly when the population growth rate is small $\rho \leq 1/r$, i.e. in the weak survival case, the particles eventually leave finite sets: the population is wandering. Secondly when the population grows moderately $1/r<\rho\leq \e^{\avez}$, particles are spread out and typically do not gather into (exponentially big) clusters: the population is rural. Finally, when the population growth rate is high $\rho>\e^{\avez}$, the particles cluster sufficiently: the population is urban.

We want to emphasize that establishing the existence of the asymptotic entropy of the empirical distributions is already a highly non-trivial task. Tools such as Kingman's subadditive theorem are not available for the asymptotic entropy of the empirical distributions, and in proving Theorem \ref{thm: main 1} we have to control the whole trajectory of the particles alive at generation $n$.

We also prove a similar result in the case where we eliminate all the dependencies of the random walks imposed by the Galton-Watson process. We consider $\rho^n$ (this is the expected number of particles in the $n$-th generation of a branching process) independent random walks starting at the identity of $G$. We believe that this result might be of independent interest.

\begin{theorem}[Independent random walks]\label{thm: main 2}
 For each $n \in \mathbb{N}$, let $\{X_{n,k}\}_{k=1}^{\rho^n}$ be i.i.d. $G$-valued random variables with common law $X_{n,1} \sim \mu^{*n}$. Define the empirical measure
$$\indempirical_n(x) := \frac{1}{\rho^n}\sum_{k=1}^{\rho^n} \mathds1_{\{X_{n,k} = x\}}, \qquad x \in G,$$
and $H(\indempirical_n)$ as in \eqref{eq:shannon-emp-distr} with $\indempirical_n(x)$ instead of $\empirical_n$.
If $\mu$ has  finite entropy and $\rho>1$, then the following holds. 
\begin{enumerate}[label = (\roman*)]
\setlength\itemsep{0em}
    \item \textbf{Small-$\rho$ regime.} If $\rho\leq \e^{\avez}$, then almost surely
    $$\lim_{n\to\infty}\frac{H(\indempirical_n)}{n}=\log\rho.$$
    \item  \textbf{Large-$\rho$ regime.} If $\rho> \e^{\avez}$, then almost surely
    $$\lim_{n\to\infty}\frac{H(\indempirical_n)}{n} = \avez.$$
\end{enumerate}
\end{theorem}
Remarkably, the asymptotic entropy has the same critical threshold $\mathrm{e}^{\avez}$ for both the BRW’s empirical distribution and the empirical distribution of $\rho^{n}$ independent random walks. Although it is often the case that branching random walks are well approximated by systems of independent particles, it is still noteworthy that offspring produced at atypical vertices (i.e., vertices where $\mu^{*n}$ takes atypical values) do not affect the limiting entropy. This shows that, in the limit, the vast majority of particles in a branching random walk follow typical trajectories, along which $\mu^{*n}$ remains close to $\mathrm{e}^{-n\avez}$.

\paragraph{Proof ideas.} The main challenge in proving the existence and the value of the asymptotic entropy in Theorem \ref{thm: main 1} is controlling the influence of offspring produced in earlier generations at vertices $x\in G$ where $\mu^{*k}$ takes atypical values, on the Shannon entropy of the empirical distribution in later generations. Although this influence can be uniformly controlled when the expected offspring number $\rho$ is sufficiently small (specifically, below the inverse spectral radius), the argument fails once $\rho$ becomes too large. To overcome this, we introduce the entropic tube: we restrict our attention to branching particles whose entire lineage behaves typically. Concretely, for each particle $v\in\mathbb{T}_n$ at generation $n$ and its ancestor $v_k$ at generation $k$, we require that $-\log(\mu^{*k}(X_{v_k}))$ stays within a small perturbation of $k\avez$. Restricting the empirical distribution to particles inside this entropic tube provides the control needed to establish the existence of the asymptotic entropy and identify its value for the restricted process. We then complete the proof of Theorem \ref{thm: main 1} by showing that, with a suitable choice of perturbation and bounds on the restricted population growth, the unrestricted asymptotic entropy is well approximated by the restricted one.

In the independent-particle setting of Theorem \ref{thm: main 2}, this difficulty does not arise, since atypical behaviour of one particle has no effect on the other particles. It is therefore enough to consider the positions after $n$ steps rather than full trajectories. We first restrict our attention to those particles for which $\mu^{*n}(X_{n,k})$ is sufficiently close to $e^{-n\avez}$. In the small-$\rho$ regime we show that no clustering occurs, while in the large-$\rho$ regime the empirical distribution is close to $\mu^{*n}$. These results yield the desired statements for the typical (restricted) particles, and we then extend them to all particles by proving that the number of atypical particles is negligible in the limit.

\begin{remark}
    If $G$ is a finite group, than $\avez=0$ and the small-$\rho$ regime does not appear, while in the large-$\rho$ regimes of Theorem \ref{thm: main 1}(ii) and Theorem \ref{thm: main 2}(ii), the statements are trivial, that is,
    almost surely 
    \[
    \lim_{n\to\infty}\frac{H(\empirical_n)}{n} =0 \quad \text{and} \quad\lim_{n\to\infty}\frac{H(\indempirical_n)}{n} = 0.
    \]
\end{remark}

\textbf{Structure of the paper.} In Section \ref{sec: prelim} we introduce Galton-Watson trees, branching random walks, and empirical distributions. In Section \ref{sec: gen_results} we prove some auxiliary results needed for later proofs. In Section \ref{sec:ind-rw} we prove Theorem \ref{thm: main 2}, while in Section \ref{sec:brw} we prove Thereom \ref{thm: main 1}. We conclude with some open questions.

\section{Preliminaries}\label{sec: prelim}

In this section, we rigorously define branching random walks and Shannon entropy, and we state the assumptions on the offspring and step distributions that will be used in the proofs of our main theorems.

\textbf{Random walks on finitely generated groups.} Consider a countable group $G$ with neutral element $e$, and the (right) random walk $(Y_n)_n$ on $G$ with step distribution $\mu\in\textsf{Prob}(G)$: for every $n\in\N$, $Y_n = \xi_1\xi_2\cdots\xi_n$ where $(\xi_k)_k$ is an i.i.d.~sequence of $G$-valued random variables with distribution $\mu$. So $(Y_n)_n$ is a Markov chain on $G$ with transition matrix $P=(\mu(x^{-1}y))_{x,y\in G}$, and the $n$-step transition probabilities are: for $n\in \N$ and $x,y\in G$
$$\Prob(Y_n = y|Y_0=x) = \mu^{*n}(x^{-1}y)$$
where $\mu^{*n}$ is the $n$-fold convolution of $\mu$.

\textbf{Entropy.} For a probability measure $\nu\in\textsf{Prob}(G)$, its Shannon entropy is defined as
$$\entropy(\nu):= -\sum_{g\in G}\nu(g)\log\nu(g)$$
as long as the sum converges. If for the step distribution $\mu$ of the random walk on $G$, the entropy $H(\mu)$ is finite, then 
$$\avez:=\lim_{n\to\infty}\frac{H(\mu^{*n})}{n} = \lim_{n\to\infty}\frac{-\E[\log\mu^{*n}(Y_n)]}{n} < \infty,$$
where the limit exists due to subadditivity of $(H(\mu^{\ast n}))_{n\in \mathbb{N}}$ and Fekete's lemma. We call this quantity the Avez entropy (also asymptotic entropy).
A similar limit holds also pathwise by the Shannon type theorem \cite[Theorem 2.1]{KV83}, which states that for $\mathbb{P}$ almost every trajectory
$$\avez = \lim_{n\to\infty}\frac{-\log\mu^{*n}(Y_n)}{n}.$$

\textbf{Galton-Watson trees.}
A Galton-Watson process describes the random reproduction of particles, and the associated Galton-Watson tree $\GWT$ captures the whole genealogy. For an offspring distribution $\pi\in\textsf{Prob}(\N_0)$ we define the Galton-Watson tree inductively. Start with a root $\emptyset$. If up to generation $n$ the tree is defined, we sample for every leaf at level $n$, a random variable distributed according to $\pi$ independently of everything else, and then attach the sampled number of offspring to the level $n$ leaves. For a vertex $v\in\GWT$ we write $|v|=n$ if it is a particle of generation $n$, and for $j\leq n$ we write $v_j$ for its ancestor in generation $j$. Set 
$$\GWT_n = \{v\in\GWT: \, |v|=n\},\quad \text{and}\quad Z_n=\#\GWT_n,$$
the number of particles alive in generation $n\in\N$. It is well known (see~\cite{Harris-book}) that if the mean offspring number $\rho:=\sum_{k=0}^\infty k\pi(k)$
is $\leq 1$ the population dies out almost surely, i.e. the Galton-Watson tree is almost surely finite, and if $\rho>1$ the population survives with positive probability. If the $L\log L$ moment condition for the offspring distribution holds, that is $\sum_{k=0}^\infty \pi(k)k\log k < \infty,$
then the population martingale $W_n=Z_n/\rho^n$ has an almost sure limit $W$, which in view of Kesten-Stigum \cite{Kesten-Stigum} is almost surely finite and strictly positive on the event of survival. 

\textbf{Branching random walks (shortly BRW).} A branching random walk on $G$ with offspring distribution $\pi\in\textsf{Prob}(\N)$ and step distribution $\mu\in\textsf{Prob}(G)$, shortly written as $\BRW(G,\pi,\mu)$, is defined as the tree indexed random walk $(X_v)_{v\in\GWT}$, where $\GWT$ is the Galton-Watson tree with offspring distribution $\pi$. For every vertex $v$ in $\GWT$ we assign a $G$-valued random variable $\zeta_v\sim \mu$, such that $(\zeta_v)_{v\in \GWT}$ are i.i.d.
The $G$-valued random variable $X_v$ for some $v\in\GWT$ with $|v|=k$ is given as 
$$X_v= \zeta_\emptyset\zeta_{v_1}\zeta_{v_2}\cdots\zeta_{v_{k-1}}\zeta_v,$$
where $v_j$ is the ancestor of $v$ in generation $j\leq k$ and $\zeta_\emptyset$ is the position of the root particle. If not otherwise mentioned, we take $\zeta_\emptyset = e$. There is a phase transition in the mean offspring number $\rho$, in what concerns the behavior of a BRW.
The process $\BRW(G,\pi,\mu)$ survives weakly if the population survives, but eventually every finite subset of $G$ gets vacated, i.e.\ for every finite $K\subset G$ it holds
$$\Prob(\exists n_0\in\N:\text{ for all } n\geq n_0 \text{ and } v\in\GWT_n,\  X_v\notin K) = 1.$$
Otherwise the branching random walk survives strongly. In fact, it holds that for $\rho\leq 1/r$ the branching random walk has a weak survival phase, and for $\rho>1/r$ the branching random walk survives strongly; see Benjamini-Peres \cite{recurrence-brw-benjamini-peres-1994}.

\textbf{Empirical distribution.}
We define the empirical distribution of the branching random walk $\BRW(G,\pi,\mu)$ with offspring distribution $\pi\in\textsf{Prob}(\N)$ and step distribution $\mu\in\textsf{Prob}(G)$ on a countable group $G$ as the random probability measure on $G$ defined as 
\begin{equation}\label{eq:emp-distr-brw}
  \empirical_n(x) :=\frac{M_n(x)}{Z_n}.
\end{equation}
where $Z_n = \#\GWT_n$ is the Galton-Watson process associated with the branching random walk $\BRW(G,\pi,\mu)$ and $M_n(x) = \sum_{v\in\GWT_n}\mathds{1}_{\lbrace X_v = x\rbrace}$.

\textbf{Assumptions}
We introduce here the assumptions on both the underlying random walk and the Galton–Watson process under which our results hold. We consider only supercritical branching processes, that is, $\rho>1$ during this paper. 
\begin{enumerate}[label = (A\arabic*)]
\setlength\itemsep{0em}
    \item\label{ass: finite entropy} The step distribution $\mu$ has finite entropy, i.e. $-\sum_{x\in G}\mu(x)\log\mu(x)<\infty$.
    \item\label{ass: almost sure survival} Particles always have a positive number of offspring, i.e. $\pi(0)=0$.
    \item\label{ass: sec mom GW} The offspring distribution $\pi$ has finite second moment, i.e.~$\sum_{k=1}^\infty \pi(k)k^2<\infty$. 
\end{enumerate}
Note that under assumptions \ref{ass: almost sure survival} and \ref{ass: sec mom GW} the mean offspring number $\rho=\E\pi$ is finite and the population martingale $W_n = Z_n/\rho^n$ has an almost sure positive and finite limit $W$ in view of \cite{Kesten-Stigum}.

\section{Auxiliary results}\label{sec: gen_results}
We first establish several auxiliary results needed later, and we begin with a general statement about the Shannon entropy.

\begin{lemma}\label{lem: Lemma on entropy}
Let $\mu,\nu,\mu_1,\mu_2,\dots,\mu_n$ be probability measures on $G$, where $n\in\N$. The following statements hold\rk{:} 
\begin{enumerate}[label = (\roman*)]
    \item Let $\lambda_1,\lambda_2,\dots,\lambda_n\geq 0$ with $\lambda_1+\lambda_2+\cdots+\lambda_n=1$. Then it holds
    $$
    \sum_{k=1}^n\lambda_k H(\mu_k)\leq H\Big(\sum_{k=1}^n\lambda_k\mu_k\Big)\leq \sum_{k=1}^n\lambda_k H(\mu_k) + \log n.
    $$
    \item Let $c\in(0,1)$. If for all $x\in\N$ we have $\mu(x)\leq c$, then $H(\mu)\geq -\log c$.
    \item  Let $c\in (0,1)$. If for all $x\in\textsf{supp}(\mu)$ we have $\mu(x)\geq c$, then $ H(\mu)\leq -\log c.$
    \item Let $0\leq c<1<C\leq \infty$, and $c\nu\leq \mu \leq C\nu$ pointwise. Then it holds
    $$
    cH(\nu)-\log C\leq H(\mu) \leq CH(\nu)-\log c.
    $$
\end{enumerate}
\begin{proof}
We start with $(i)$. The first inequality is due to the concavity of $x\mapsto -x\log x$. For the second inequality we have 
\begin{align*}
H\Big(\sum_{k=1}^n &\lambda_k\mu_k\Big)  = -\sum_{x\in G}\Big(\sum_{k=1}^n\lambda_k\mu_k(x)\Big)\log\Big(\sum_{j=1}^n\lambda_j\mu_j(x)\Big)
=-\sum_{k=1}^n\lambda_k\sum_{x\in G}\mu_k(x)\log\Big(\sum_{j=1}^n\lambda_j\mu_j(x)\Big)\\
&\leq -\sum_{k=1}^n\lambda_k\sum_{x\in G}\mu_k(x)\log(\lambda_k\mu_k(x))
= \sum_{k=1}^n\lambda_kH(\mu_k) - \sum_{k=1}^n\lambda_k\log\lambda_k
\leq \sum_{k=1}^n\lambda_k H(\mu_k) + \log n.
\end{align*}
Above we have used again the concavity of $-x\log x$ to conclude that $-\sum_{k=1}^n\lambda_k\log\lambda_k\leq\log n$. Since  $-\log$ is decreasing, $(ii)$ and $(iii)$ follow directly. The last claim $(iv)$ is again a simple calculation. For the first inequality we have
\begin{align*}
H(\mu)=-\sum_{x\in G}\mu(x)\log\mu(x) \geq -\sum_{x\in G}\mu(x) \log(C\nu(x)) \geq -\log C + \sum_{x\in G}c\nu(x)(-\log\nu(x)).
\end{align*}
 The second inequality works exactly the same way, by switching the roles of $c$ and $C$.
\end{proof}
\end{lemma}

We show  next that the expectation of the empirical distribution equals the  $n$-step  transition probability of the underlying random walk $(G,\mu)$.
\begin{lemma}\label{lem: expectation of empirical}
For every $n\in\N$ it holds
$\E\big[\empirical_n\big] = \mu^{*n}$.
\end{lemma}
\begin{proof}
    Choosing $A\subset G$ and conditioning on $\GWT$ we obtain
    \begin{align*}
    \E\big[\empirical_n(A)\,|\,\GWT\big] &= \E\Big[\frac{1}{Z_n}\sum_{v\in \GWT_n} \mathds1_{ \{X_v\in A\}}\,|\,\GWT\Big] = \frac{1}{Z_n}\sum_{v\in \GWT_n}\E\big[\mathds1_{ \{X_v\in A\}}\,|\,\GWT\big] \\
    &=\frac{1}{Z_n}\sum_{v\in\GWT_n}\Prob(Y_n\in A) = \mu^{*n}(A).
    \end{align*}
    Here we used that the random walk steps are independent of the Galton–Watson process. Taking the expectation yields the result.
\end{proof}
We apply this lemma to prove that the expected entropy of the empirical distribution is always bounded above by the entropy of the random walk’s transition probability.
\begin{lemma}\label{lem: expectation of entropy}
For every $n\in\N$ we have
$\E\big[H(\empirical_n)\big]\leq H(\mu^{*n})$.
\end{lemma}
\begin{proof}
The function $[0,1]\to [0,\infty):x\mapsto -x\log x$
is concave, so due to Jensen's inequality
\begin{align*}
    \E\big[H(\empirical_n)\big] &= \sum_{x\in G} \E\big[-\empirical_n(x)\log\empirical_n(x)\big] \leq \sum_{x\in G} -\E\big[\empirical_n(x)\big]\log\E\big[\empirical_n(x)\big]=H(\E\big[\empirical_n\big])=H(\mu^{*n}),
\end{align*}
where the last equality uses Lemma \ref{lem: expectation of empirical}.
\end{proof}
Lastly, we show that the asymptotic entropy of the branching random walk can never exceed $\log\rho.$
\begin{lemma}\label{lem: trivial upper}
Under assumptions \ref{ass: almost sure survival} and \ref{ass: sec mom GW} it holds
$$
\limsup_{n\to\infty}\frac{H(\empirical_n)}{n}\leq \lim_{n\to\infty}\frac{\log Z_n}{n} = \log\rho
$$
almost surely.
\end{lemma}
\begin{proof}
The first inequality is trivial since the Shannon entropy is maximized by the uniform distribution on the support and the size of the support of $\empirical_n$ is less or equal than the total number $Z_n$ of particles. Assumptions \ref{ass: almost sure survival} and \ref{ass: sec mom GW} imply that the limit $W$ of the population martingale $W_n = Z_n/\rho^n$ exists almost surely and is positive with probability $1$. Therefore
$$
\lim_{n\to\infty}\frac{\log Z_n}{n} = \lim_{n\to\infty}\frac{n\log\rho+\log W_n}{n} = \log \rho,
$$
and this proves the claim.
\end{proof}

\section{Independent random walks}\label{sec:ind-rw}


In this section we prove Theorem \ref{thm: main 2}. We partition the collection of $\rho^n$ particles (i.e.,  $\rho^n$ independent random walks starting at the identity of $G$) into typical and atypical ones, where “typical” refers to their entropic behaviour.  We show that for $\rho \le \e^{\avez}$ the particles do not cluster too much, whereas for $\rho > \e^{\avez}$ the asymptotic entropy of the system of typical particles resembles the entropy of a single random walk. We emphasize that, in the independent-particle setting, we actually obtain a stronger result about the empirical distributions of typical particles. For $\rho \le \e^{\avez}$, we prove an $L^\infty$ bound on clustering, showing that the cluster size grows slower than any exponential. For $\rho > \e^{\avez}$, we show that the empirical distribution of typical particles approximates the $n$-step distribution $\mu^{*n}$ arbitrarily well in the $L^\infty$ sense. These statements are collected in Theorem \ref{thm: main ind empirical} and yield convergence of the asymptotic entropy in the independent case. We believe such $L^\infty$ control of empirical distributions is of independent interest.

\subsection{Typical and odd particles}
For fixed $\delta>0$ and $n\in\N$ define the set of typical particles
$$
\typind_n:=\typind_{n,\delta}:=\big\{k\in\{1,2,\dots,\rho^n\}: |-\log\mu^{*n}(X_{n,k})-n\avez|\leq \delta n\big\},
$$
and the set of odd particles as 
$$
\oddind_n:=\oddind_{n,\delta} := \{1,2,\dots,\rho^n\}\setminus \typind_{n,\delta}.
$$
These two sets induce a natural partition of the empirical distribution $\indempirical_n$ into the empirical distribution of the typical and of the odd particles respectively, thus 
$$
\indempirical_n = \frac{\#\typind_n}{\rho^n}\indempirical_n^{\text{typ}}+\frac{\#\oddind_n}{\rho^n}\indempirical_n^{\text{odd}},
$$
where we set 
\begin{align*}
\indempirical_n^{\text{typ}}(x) = \frac{1}{\#\typind_n}\sum_{k\in \typind_n}\mathds1_{\{X_{n,k}=x\}}
\hspace{1cm}\text{and}\hspace{1cm}
\indempirical_n^{\text{odd}}(x) = \frac{1}{\#\oddind_n}\sum_{k\in \oddind_n}\mathds1_{\{X_{n,k}=x\}}.
\end{align*}
We write $\mathcal{M}_n^{\text{typ}}(x) = \#\typind_n\cdot\indempirical_n^{\text{typ}}(x)$ as well as $\mathcal{M}_n^{\text{odd}}(x) = \#\oddind_n \cdot\indempirical_n^{\text{odd}}(x)$ for the number of typical and odd particles at site $x\in G$ in generation $n$, respectively.

\begin{lemma}\label{lem: ind RW untypical goes to 0}
 Under Assumption \ref{ass: finite entropy}, it holds
 $$
 \lim_{n\to\infty}\frac{\#\typind_n}{\rho^n} = 1,\qquad \text{almost surely}.
 $$
 In particular, it holds
  $$
 \lim_{n\to\infty}\frac{H\big(\indempirical_n^{\text{typ}} \big)}{H\big(\indempirical_n \big)} = 1,\qquad \text{almost surely}.
 $$
\end{lemma}
\begin{proof} 
We first show that the ratio $\#\typind_n/\rho^n$ converges to $1$  almost surely, as $n$ goes to infinity.
For this we prove that $\lim \#\oddind_n/\rho^n = 0$, by showing that for every $\varepsilon>0$ there exist $n_0\in\N$ and  $a<1$ such that 
$$
\Prob(\#\oddind_n\geq \rho^n\varepsilon)\leq a^{\rho^n}
$$
for every $n\geq n_0$. Then the claim follows from the Borel-Cantelli lemma. Set
\[
B_n = \{|-\log\mu^{*n}(Y_n)-n\avez|>\delta n\}\,.
\]
Then due to Shannon's theorem \cite[Theorem 2.1]{KV83}, we get that $\lim_n\Prob(B_n)=0$, which together with the Chernoff bound yields for $t>0$ that
$$
\Prob\big(\#\oddind_n\geq \rho^n\varepsilon\big) \leq \bigg(\frac{\E\big[\e^{t\mathds1_{B_n}}\big]}{\e^{t\varepsilon}}\bigg)^{\rho^n} = \bigg(\frac{1-\Prob(B_n)+\e^{t}\Prob(B_n)}{\e^{t\varepsilon}}\bigg)^{\rho^n}.
$$
We choose $n_0$, such that $\Prob(B_n) < \varepsilon/2$ for any $n\geq n_0$. Then linearization near $0$ gives the existence of a $t>0$ such that 
$$
1-\frac{\varepsilon}{2}+\e^t\frac{\varepsilon}{2} < \e^{t\varepsilon},
$$
which combined with the previous inequality completes the proof.

The second statement of the lemma follows from Lemma \ref{lem: Lemma on entropy} which yields
\begin{equation}\label{eq: lower upper indempirical}
     \begin{aligned}
     \frac{\#\typind_n}{\rho^n}H\big(\indempirical_n^{\text{typ}}\big)&\leq H\big(\indempirical_n\big)\leq
      \frac{\#\typind_n}{\rho^n}H\big(\indempirical_n^{\text{typ}}\big)+\frac{\#\oddind_n}{\rho^n}H\big(\indempirical_n^{\text{odd}}\big)+\log 2\\
     &\leq \frac{\#\typind_n}{\rho^n}H\big(\indempirical_n^{\text{typ}}\big)+\frac{n\log\rho\cdot \#\oddind_n}{\rho^n}+\log 2.
     \end{aligned}
\end{equation}
\end{proof}

Due to Lemma \ref{lem: ind RW untypical goes to 0}, to prove Theorem \ref{thm: main 2} it is enough to prove the following result.

\begin{theorem}\label{thm: main ind empirical}
Under assumption \ref{ass: finite entropy}, the two following statements hold. 
\begin{enumerate}[label = (\roman*)]
    \item Let $\rho \leq \e^{\avez}$. Then, for any $c>0$, there exists $\delta_0>0$, such that for all $\delta\leq\delta_0$ it holds that 
    $$
    \Vert\mathcal{M}_n^{\text{typ}}\Vert_\infty \leq (1+c)^n
    $$ 
   almost surely for  all sufficiently large $n\in\N$.
    \item Let $\rho>\e^{\avez}$. Then, for any $0\leq c < 1 < C \leq \infty$ there exists $\delta_0>0$, such that for all $\delta\leq \delta_0$ it holds that 
    $$
    \indempirical_n^{\text{typ}}(x)\leq C\mu^{*n}(x), \hspace{1cm}x\in G,
    $$
    and that 
    $$
    \indempirical_n^{\text{typ}}(x)\geq c\mu^{*n}(x),\hspace{1cm}x\in\textsf{supp}(\indempirical_n^{typ}),
    $$
    almost surely for  all sufficiently large $n\in\N$. 
\end{enumerate}
\end{theorem}
Before we prove the theorem, we show that Theorem \ref{thm: main 2} is a simple consequence of Theorem \ref{thm: main ind empirical}.

\begin{proof}[Proof of Theorem \ref{thm: main 2}]
Using Equation \ref{eq: lower upper indempirical} and Lemma \ref{lem: ind RW untypical goes to 0}, it suffices to prove the two claims in Theorem \ref{thm: main 2} for $\indempirical_n^{\text{typ}}$. First let $\rho\leq \e^{\avez}$. Then, by Theorem \ref{thm: main ind empirical} and Lemma \ref{lem: Lemma on entropy}, for every $c>0$ there exists some $\delta>0$ such that we have, almost surely, for all sufficiently large $n\in\N$, 
$$
n\log\rho+\log\Big(\frac{\#\typind_n}{\rho^n}\Big)-n\log(1+c)=\log\big(\#\typind_n\big)-n\log(1+c) \leq H(\indempirical_n^{\text{typ}}) \leq n\log\rho.
$$
Since we can choose $c>0$ arbitrarily, and due to Lemma \ref{lem: ind RW untypical goes to 0}, the small-$\rho$ regime is proven. Next, let $\rho>\e^{\avez}$. Then if we define $\mathcal{A}_n:=\big\{x\in G:\,n(\avez-\delta)\leq -\log\mu^{*n}(x)\leq n(\avez+\delta)\big\}$ it holds
$\textsf{supp}(\indempirical_n^{\text{typ}})\subset \mathcal{A}_n$, and therefore
$$
H(\indempirical_n^{\text{typ}})\leq \log(\#\mathcal{A}_n).
$$
Moreover, 
$$
\e^{n(\avez+\delta)} \geq \sum_{x\in\mathcal{A}_n}\e^{n(\avez+\delta)}\mu^{*n}(x) \geq \#\mathcal{A}_n,
$$
and so $H(\indempirical_n^{\text{typ}})\leq n(\avez+\delta)$, which finishes the upper bound. For the lower bound we get 
$$
-\sum_{x\in G}\indempirical_n^{\text{typ}}(x)\log\indempirical_n^{\text{typ}}(x) \geq C(\avez-\delta)\sum_{x\in G}\indempirical_n^{\text{typ}}(x) = C(\avez-\delta),
$$
and since $C>1$ can be chosen arbitrarily close to $1$, the proof is finished.
\end{proof}
We now split the proof of Theorem \ref{thm: main ind empirical} into two parts, the small-$\rho$ and the large-$\rho$ regime.

\subsection{Small-$\rho$ regime}
\begin{proof}[Proof of Theorem \ref{thm: main ind empirical}{(i)}]
Recall the definition of  $\mathcal{M}_n^{\text{typ}}$, which counts the number of typical particles at time $n$. For some $1>\delta>0$, define the set 
\[
\mathcal{A}_n:=\mathcal{A}_{n,\delta}:=\big\{x\in G:\,n(\avez-\delta)\leq -\log\mu^{*n}(x)\leq n(\avez+\delta)\big\}\,.
\]
Note that for all $x \in \mathcal{A}_n$, we have $\mu^{*n}(x) \ge e^{-n(\avez+1)}$, so 
\begin{align}\label{cardAn}
\#\mathcal{A}_n \le e^{n(\avez+1)}.
\end{align}
The definition of $\typind_n$ implies that $\textsf{supp}(\mathcal{M}_n^{\text{typ}}) \subset \mathcal{A}_n$, which together with a union bound implies that 
\begin{align*}
\Prob\big(\text{there exists } x\in G \text{ such that } \mathcal{M}_n^{\text{typ}}(x)\geq (1+c)^n\big) &\leq \sum_{x\in  \mathcal{A}_n}\Prob\big(\mathcal{M}_n^{\text{typ}}(x)\geq (1+c)^n\big)\\
&\leq \sum_{x\in \mathcal{A}_n}\Prob\big(\mathcal{M}_n(x)\geq (1+c)^n\big).
\end{align*}
We bound the above probabilities uniformly in $x$. The Markov inequality yields
$$
\Prob\big(\mathcal{M}_n(x)\geq (1+c)^n\big) \leq \frac{\E\big[2^{\mathcal{M}_n(x)}\big]}{2^{(1+c)^n}},
$$
and the expectation on the right-hand side can be computed explicitly, since $X_{n,1},X_{n,2},\dots,X_{n,\rho^n}$ are i.i.d. random walks distributed as $Y_n$. We have
\begin{align*}
    \E\big[2^{\mathcal{M}_n(x)}\big] &= \E\Big[2^{\sum_{k=1}^{\rho^n}\mathds{1}_{\{X_{n,k} = x\}}}\Big] = \E\Big[\prod_{k=1}^{\rho^n}2^{\mathds{1}_{\{X_{n,k}=x\}}}\Big] 
    = \prod_{k=1}^{\rho^n}\E\big[2^{\mathds{1}_{\{Y_n = x\}}}\big] \\
    &= \big(1-\mu^{*n}(x)+2\mu^{*n}(x)\big)^{\rho^n}
    \leq \big(1+\e^{-n(\avez-\delta)}\big)^{\rho^n}\leq \e^{(\rho\e^{-(\avez-\delta)})^n}\leq \e^{(\e^{\delta n})},
\end{align*}
due to the fact that for any $x\geq 0$ it holds that $1+x\leq \e^x$, and the fact that in the small-$\rho$ regime we have $\e^{-\avez}\rho\leq 1$. After choosing $\delta_0 < \log(1+c)$, we obtain for any $\delta\leq \delta_0$ that 
$$
\sum_{x\in\mathcal{A}_n} \Prob(\mathcal{M}_n(x)\geq (1+c)^n)\leq \#\mathcal{A}_n\frac{\e^{(\e^{\delta_0 n})}}{2^{((1+c)^n)}},
$$
where the right hand side is summable, since 
$$
\lim_{n\to\infty} \frac{1}{n}\log\Big(\frac{\e^{(\e^{\delta_0 n})}}{2^{((1+c)^n)}}\Big) = \lim_{n\to\infty} \frac{\e^{\delta_0 n}}{n} - \frac{\log2\cdot(1+c)^n}{n} = -\infty,
$$
which together with inequality~(\ref{cardAn}) and Borel-Cantelli lemma concludes the proof.
\end{proof}

\subsection{Large-$\rho$ regime}
\begin{proof}[Proof of Theorem \ref{thm: main ind empirical}(ii)]
Assume that $\rho>\e^{\avez}$ and choose $\delta < \log\rho - \avez$. We prove the two inequalities separately and start with the upper bound for $\indempirical_n^{\text{typ}}$. We fix constants $c<1$ and $C>1$, and prove that the probabilities $\Prob(\text{there exists }x\in\textsf{supp}(\mathcal{M}_n^{\text{typ}})\text{ such that }\mathcal{M}_{n}^{\text{typ}}(x)\leq c\rho^n\mu^{*n}(x))$ and $\Prob(\text{there exists }x\in G\text{ such that }\mathcal{M}_{n}^{\text{typ}}(x)\geq C\rho^n\mu^{*n}(x))$ form both summable sequences. This together with the Borel-Cantelli Lemma and Lemma \ref{lem: ind RW untypical goes to 0} completes the proof.

We start with the first sequence of probabilities.  Once again, we write
$$\mathcal{A}_n:=\{x\in G:\,n(\avez-\delta)\leq -\log\mu^{*n}(x)\leq n(\avez+\delta)\},$$ and from the definition of $\typind_n$  it holds $\textsf{supp}(\indempirical_n^{\text{typ}}) \subset \mathcal{A}_n$, and inequality~(\ref{cardAn}) is still valid. Moreover, from the definition of $\typind_n$ and $\mathcal{A}_n$ it holds for any $k\in\{1,2,\dots,\rho^n\}$ that $k$ is typical if and only if $X_{n,k}\in\mathcal{A}_n$, and so it holds for any $x\in\mathcal{A}_n$ that $\mathcal{M}_n(x)=\mathcal{M}_n^{\text{typ}}(x)$. A union bound  implies
\begin{align*}
    \Prob\big(\text{there exists } x\in\mathcal{A}_n \text{ such that } \mathcal{M}_n^{\text{typ}}(x)\leq c\rho^n\mu^{*n}(x)\big) &\leq  \sum_{x\in \mathcal{A}_n}\Prob\big(\mathcal{M}_n^{\text{typ}}(x)\leq c\rho^n\mu^{*n}(x)\big)\\
&= \sum_{x\in \mathcal{A}_n}\Prob\big(\mathcal{M}_n(x)\leq c\rho^n\mu^{*n}(x)\big).
\end{align*}
We now use the Markov inequality for the decreasing function $x\mapsto \e^{-tx}$, for some $t>0$, which will later be chosen in an optimal way. We obtain
$$
\Prob\big(\mathcal{M}_n(x)\leq c\rho^n\mu^{*n}(x)\big) \leq \frac{\E\big[\e^{-t\mathcal{M}_n(x)}\big]}{\e^{-t c\rho^n\mu^{*n}(x)}}.
$$
We again evaluate the expectation in the numerator directly. Independence of the random walks yields
$$
\E\big[\e^{-t\mathcal{M}_n(x)}\big] = \prod_{k=1}^{\rho^n} \E\big[\e^{-t\mathds1_{\{X_{n,k}=x\}}}\big] = (1-\mu^{*n}(x)+\mu^{*n}(x)e^{-t})^{\rho^n},
$$
and so
\begin{equation}\label{eq: ind large rho 1}
\frac{\E\big[\e^{-t\mathcal{M}_n(x)}\big]}{\e^{-tc\rho^n\mu^{*n}(x)}} = \Big((1-\mu^{*n}(x)+\mu^{*n}(x)e^{-t})\e^{tc\mu^{*n}(x)}\Big)^{\rho^n}.
\end{equation}
For fixed $\lambda>0$, the function 
$$
[0,\infty)\to[0,\infty):t\mapsto (1-\lambda+\lambda\e^{-t})\e^{c\lambda t}
$$
attains its minimum at $t_* = \log\big(\frac{c\lambda-1}{c\lambda-c}\big)$, and the corresponding minimum is given by 
$$
\Phi(\lambda):=(1-\lambda+\lambda\e^{-t_*})\e^{c\lambda t_*} = \Big(1-\lambda+\lambda\frac{c\lambda-c}{c\lambda-1}\Big)\Big(\frac{c\lambda-1}{c\lambda-c}\Big)^{c\lambda}.
$$
Computing the derivative at $0$ yields
$$
\Psi(c):=\Phi'(0)= \frac{\left(\frac{c{\lambda} - 1}{c{\lambda} - c}\right)^{c{\lambda}} \, \left(\left(c{\lambda} - c\right) \log\left(\frac{c{\lambda} - 1}{c{\lambda} - c}\right) - c + 1\right)}{c{\lambda} - 1} \Big|_{\lambda=0} = -c\log c +c -1 < 0,
$$
and the inequality on the right hand side holds as long as $c<1$. This is the case since $\Psi$ is concave and takes its maximum at $c=1$ with $\Psi(1)=0$. This implies that 
$$
\Phi(\lambda) = 1+\Psi(c)\lambda+\mathcal{O}(\lambda^2),
$$
and so for $\lambda>0$ sufficiently small there is a $\gamma>0$, such that $\Phi(\lambda) \leq 1-\gamma\lambda$,
which implies that for $n$ sufficiently large it holds
$$
\Prob\big(\mathcal{M}_n(x)\leq c\rho^n\mu^{*n}(x)\big) \leq \big(1-\gamma\mu^{*n}(x)\big)^{\rho^n}\leq \e^{-\gamma(\mu^{*n}(x)\rho^n)} \leq \e^{-\gamma\big(\e^{-\avez+\delta}\rho\big)^n}.
$$
So it holds that
\[
\Prob\big(\text{there exists } x\in \mathcal{A}_n \text{ such that } \mathcal{M}_n^{\text{typ}}(x)\leq c\rho^n\mu^{*n}(x)\big)  \le \#\mathcal{A}_n \e^{-\gamma\big(\e^{-\avez+\delta}\rho\big)^n}.
\]
Choosing $\delta < \avez-\log\rho$, using inequality~(\ref{cardAn}) and applying the Borel-Cantelli lemma finishes the first part of the proof.

For the second part, a union bound  yields
\begin{align*}
    \Prob\big(\text{there exists } x\in G \text{ such that } \mathcal{M}_n^{\text{typ}}(x)\geq C\rho^n\mu^{*n}(x)\big) &\leq \sum_{x\in \mathcal{A}_n}\Prob\big(\mathcal{M}_n^{\text{typ}}(x)\geq C\rho^n\mu^{*n}(x)\big)\\
&\leq \sum_{x\in  \mathcal{A}_n}\Prob\big(\mathcal{M}_n(x)\geq C\rho^n\mu^{*n}(x)\big).
\end{align*}
Using the Markov inequality for the increasing function $x\mapsto \e^{tx}$, for some $t>0$, to be later chosen in an optimal way, we obtain 
$$
\Prob\big(\mathcal{M}_n(x)\geq C\rho^n\mu^{*n}(x)\big) \leq \frac{\E\big[\e^{t\mathcal{M}_n(x)}\big]}{\e^{tC \rho^n\mu^{*n}(x)}}.
$$
The independence of the $\rho^n$ particles gives
$$
\E\big[\e^{t\mathcal{M}_n(x)}\big] = \prod_{k=1}^{\rho^n} \E\big[\e^{t\mathds1_{\{X_{n,k}=x\}}}\big] = (1-\mu^{*n}(x)+\mu^{*n}(x)e^t)^{\rho^n}\,,
$$
thus
$$
\frac{\E\big[\e^{t\mathcal{M}_n(x)}\big]}{\e^{Ct\rho^n\mu^{*n}(x)}} = \Big(\frac{1-\mu^{*n}(x)+\mu^{*n}(x)e^t}{\e^{Ct\mu^{*n}(x)}}\Big)^{\rho^n}\,.
$$
Following the same reasoning as above, the expression inside the parenthesis attains its minimum at $t = \log\Big(\frac{C-C\mu^{*n}(x)}{1-C\mu^{*n}(x)}\Big)$ which is close to $t_*:=\log C$ since 
$\mu^{*n}(x)$ is very small. Plugging $t_*$ into the equation above, we get
$$
\frac{1-\mu^{*n}(x)+\mu^{*n}(x)e^{t_*}}{\e^{Ct_*\mu^{*n}(x)}} = 
\frac{1+\mu^{*n}(x)(C-1)}{C^{C\mu^{*n}(x)}} \leq \Big(\frac{\e^{C-1}}{C^C}\Big)^{\mu^{*n}(x)}\,,
$$
where the last inequality above uses $C-1>0$ and $1+x\leq e^x$ for $x\geq 0$. Since for $C>1$,  $e^{C-1} < C^C$, we obtain
$$
   \Prob\big(\text{there exists } x\in G \text{ such that } \mathcal{M}_n^{\text{typ}}(x)\geq C\rho^n\mu^{*n}(x)\big) \leq \#\mathcal{A}_n\cdot \Big(\frac{\e^{C-1}}{C^C}\Big)^{(\e^{-(\avez+\delta)}\rho)^n}\,,
$$
and since $\delta<\log\rho-\avez$, the sequence is summable, which finishes the proof.
\end{proof}

\section{Branching random walks}\label{sec:brw}

In this section we prove Theorem \ref{thm: main 1}, the main result of this paper. 
In the branching random walk setting, we cannot control the empirical distribution of typical particles as sharply as in the independent case (see Theorem \ref{thm: main ind empirical}). The key difference is that, for independent particles,  Chernoff bounds allow to control the probability of atypical clustering. In contrast, dependencies in branching random walks have to be dealt with. For this, we rely on the many-to-two principle to compute second moments, which we can then use to show that the set of atypical clusterings is small and negligible in the limit.

\subsection{Typical and atypical particles}

For constants $\delta,K>0$ define the set of typical particles in the $n$-th generation
$$
\typ_n := \typ_{n,\delta,K} := \big\{ v\in\GWT_n:\,\text{for every } k\leq n \text{ it holds }\big|-\log(\mu^{*k}(X_{v_k}))-k\avez \big|\leq k\delta+K\big\},
$$
and the set of odd (or atypical) particles in the $n$-th generation as 
$$
\odd_n:=\odd_{n,\delta,K}:=\GWT_n\setminus\typ_{n,\delta,K}.
$$
Analogously to the independent random walks setting, the empirical distribution can be written as a convex combination of the empirical distributions of typical and odd particles, i.e.,
$$
\empirical_n = \frac{\#\typ_n}{Z_n}\empirical_n^{\text{typ}}+\frac{\#\odd_n}{Z_n}\empirical_n^{\text{odd}},
$$
where for $x\in G$ we write
$$
\empirical_n^{\text{typ}}(x) :=\empirical_{n,\delta,K}^{\text{typ}}(x):=\frac{1}{\#\typ_{n}}\sum_{v\in\typ_{n}}\mathds1_{\{X_v=x\}},
$$
and
$$
\empirical_n^{\text{odd}}(x) :=\empirical_{n,\delta,K}^{\text{odd}}(x) :=\frac{1}{\#\odd_{n}}\sum_{v\in\odd_{n}}\mathds1_{\{X_v=x\}}.
$$
Moreover, we write $M_n^{\text{typ}}(x)=M_{n,\delta,K}^{\text{typ}}(x)$ and $M_n^{\text{odd}}(x)=M_{n,\delta,K}^{\text{odd}}(x)$ for the total number of typical/odd particles in $x\in G$, i.e. $ M_n^{\text{typ}} = \#\typ_n\cdot\empirical_n^{\text{typ}}$ and $M_n^{\text{odd}}=\#\odd_n\cdot\empirical_n^{\text{odd}}$. We call the event that defines a typical particle \emph{the entropic tube event}, which for $\delta$ and $K$ as above, we define as
$$
\ent_{n}:=\ent_{n,\delta,K}:=\big\{\text{for every } k\leq n \text{ it holds } \big|-\log\mu^{*k}(Y_k)-k\avez \big|\leq k\delta+K\big\},
$$
where we recall that $(Y_n)_{n\in\N}$ is the random walk on $G$ with step distribution $\mu$. 
The next result establishes that the entropic tube event occurs with high probability.
\begin{lemma}\label{lem: entropic-tube-small}
Under Assumption \ref{ass: finite entropy}, for any $\delta,\varepsilon>0$ there exists $K>0$ such that for all $n\in\mathbb{N}$,
$$\mathbb{P}(\ent_{n,\delta,K})\geq 1-\varepsilon.$$
\end{lemma}
\begin{proof}
We start by modifying slightly the event $\ent_{n,\delta,K}$ as follows. Let
$$
\ent_{n,\delta,N_0}':=\big\{\text{for every }N_0< k\leq n \text{ it holds }\big|-\log \mu^{*k}(Y_k)-k\avez\big|\leq k\delta\big\},
$$
where $N_0\in\N$ and consider the event
\[
H_{N_0,K}:=\big\{\text{for every }k\leq N_0\text{ it holds }\big|-\log\mu^{*k}(Y_k)-k\avez\big|\leq K\big\}\,,
\]
Then we get 
\begin{align*}
    \Prob(\ent_{n,\delta,N_0}')&=\Prob(\ent_{n,\delta,N_0}'\cap H_{N_0,K})+\Prob(\ent_{n,\delta,N_0}'\cap H_{N_0,K}^c) \\
    & \leq \Prob(\ent_{n,\delta,K}\cap H_{N_0,K})+\Prob(\ent_{n,\delta,N_0}'\cap H_{N_0,K}^c) \\
    & \leq \Prob(\ent_{n,\delta,K})+ \Prob(H_{N_0,K}^c).
\end{align*}
Markov's inequality together with a union bound implies
$$
\Prob(H_{N_0,K}^c)\leq \frac{1}{K}\sum_{k=1}^{N_0}\E\big[|-\log\mu^{*k}(Y_k)-k\avez|\big],
$$
and so it holds for any fixed $N_0$, that $\lim_{K\rightarrow\infty}\Prob(H_{N_0,K}^c)=0$. Therefore, it suffices to show that for any $\delta>0$ and any $\varepsilon>0$, there exists a $N_0$, such that for any $n\in\N$ we have $\Prob(\ent_{n,\delta,N_0}')\geq 1-\varepsilon$. First note that the events $\ent_{n,\delta,N_0}'$ are decreasing and so it suffices to show that there exists $N_0$ big enough such that $\Prob(\ent_{\delta,N_0}')\geq 1-\varepsilon$, where $\ent_{\delta,N_0}' = \{\text{for every }k> N_0 \text{ it holds }|-\log\mu^{*k}(Y_k)-k\avez|\leq k\delta\}$. This follows easily from Shannon's theorem \cite[Theorem 2.1]{KV83}, i.e. from the fact that $-\log\mu^{*k}(Y_k)/k\to \avez$ almost surely, and so the claim follows.
\end{proof}
We use this result to show that the ratio of odd particles to all particles converges and can be made arbitrarily small by choosing $K>0$ sufficiently large.
\begin{lemma}\label{lem: lemma on typical and odd particles}
The following limits exist for every $\delta,K>0$ almost surely
\begin{align*}
W^{\text{typ}}:= W_{\delta,K}^{\text{typ}}:=\lim_{n\to\infty}\frac{\#\typ_n}{\rho^n},&&W^{\text{odd}} := W_{\delta,K}^{\text{odd}} := \lim_{n\to\infty}\frac{\#\odd_n}{\rho^n}.
\end{align*}
Moreover, for the limit $W$ of the population martingale $W_n=Z_n/\rho^n$ it holds $W = W^{\text{typ}}+W^{\text{odd}}$,
and also the following limits hold almost surely
\begin{align*}
\lim_{K\to\infty} W_{\delta,K}^{\text{typ}} = W,& &\lim_{K\to\infty} W_{\delta,K}^{\text{odd}} = 0. 
\end{align*}
\end{lemma}
\begin{proof}
Since $Z_n = \#\typ_n+\#\odd_n$ it clearly suffices to prove that the limit $\lim_{n\rightarrow\infty}\#\odd_n/\rho^n$ (which depends on $K$) exists and it converges to $0$, as $K$ goes to infinity. Observe that any offspring of an odd particle is itself an odd particle. Thus, for every $n\in\N$ it holds
$$
\odd_{n+1}\supset \bigcup_{v\in\odd_n}\{w\in\GWT_{n+1}:\,v<w\},
$$
and if $\mathcal{F}_n$ is the $\sigma$-algebra generated by the branching random walk up to time $n$, then
$$
\E\left[\#\odd_{n+1}\,|\,\mathcal{F}_n\right] \geq \rho\cdot\#\odd_n,
$$
which shows that $\#\odd_n/\rho^n$ is a submartingale with bounded expectation, and so the limit exists almost surely for any $K>0$.
Notice that by the definition of the entropic tube event, the family of limits $(W^{\text{odd}}_{\delta,K})_{K>0}$ is monotonically decreasing in $K$. Thus, the limit $\lim_{K\rightarrow\infty}W^{\text{odd}}_{\delta,K}$ exists almost surely. Fatou's lemma together with Lemma \ref{lem: expectation of empirical} and Lemma \ref{lem: entropic-tube-small} yield
$$\E[\lim_{K\to\infty}W_{\delta,K}^{\text{odd}}]=\E\Big[\liminf_{K\to\infty}W_{\delta,K}^{\text{odd}}\Big]\leq \liminf_{K\to\infty}\E\Big[W_{K,\delta}^{\text{odd}}\Big] \leq \liminf_{K\to\infty}\sup_{n\in\N}\Prob(\ent_{n,\delta,K}^c) = 0,$$
which concludes the proof.
\end{proof}
The next lemma provides an upper bound on the $L^2$-norm of $M_n^{typ}$, which will be crucial for later proofs.

\begin{lemma}\label{lem: two norm typical}
Suppose Assumptions \ref{ass: finite entropy}, \ref{ass: almost sure survival}, and \ref{ass: sec mom GW} hold. Then there exists a universal constant $C > 0$ such that, for all $n\in\N$
$$
\E\Big[\sum_{x\in G}M_{n}^{\text{typ}}(x)^2\Big]\leq C\rho^n e^{2n\delta+2K}\sum_{j=0}^n\big(\e^{-\avez}\rho\big)^j.
$$
\end{lemma}
\begin{proof}
For $x\in G$ we have
$$
M_n^{\text{typ}}(x)^2=\sum_{v,w\in\mathbb{T}_n}\mathds1_{\{X_v=x,v\text{ is typical}\}}\mathds1_{\{X_w=x,w\text{ is typical}\}}.$$
We partition $\GWT_n\times\GWT_n$ as follows. Set
$$
\mathcal{N}_n^{(k)}:=\{(v,w)\in\GWT_n\times\GWT_n:\,|v\wedge w|=k\},
$$
for the pairs $(v,w)$ of vertices in the $n$-th generation with most recent common ancestor in the $k$-th generation, and write  $N_n^{(k)}$ for the cardinality of $\mathcal{N}_n^{(k)}$. It clearly holds that $(\mathcal{N}_n^{(k)})_{k\leq n}$ is a partition of $\GWT_n\times\GWT_n$. Next we estimate the expectation of $N_n^{(k)}$. For given $u\in\GWT_k$, we write $Z_n^{(u)}$ for the number of descendants of $u$ in generation $n$. Then conditioned on $\GWT_k$, we have $Z_n^{(u)}\sim Z_{n-k}$, and the random variables $(Z_n^{(u)})_{u\in\GWT_k}$ are jointly independent. For every $(v,w)\in\mathcal{N}_n^{(k)}$ it holds that $v$ and $w$ have a common ancestor in generation $k$, thus 
$$N_n^{(k)}\leq \sum_{u\in\GWT_k}\big(Z_n^{(u)}\big)^2.$$
Conditioning on $\GWT_k$ gives $\E[N_n^{(k)}|\GWT_k]\leq Z_k\E[Z_{n-k}^2]$, and consequently
$$
\E\big[N_n^{(k)}\big]\leq \rho^k \E[Z_{n-k}^2] \leq C\rho^{2n-k}
$$
for some constant $C>0$, whose existence follows from the second moment condition \ref{ass: sec mom GW}. Using that $(\mathcal{N}_n^{(k)})_{k\leq n}$ partitions $\GWT_n\times\GWT_n$ gives 
\begin{align*}
\mathbb{E}\Big[M_n^{\text{typ}}(x)^2\Big]&=\mathbb{E}\bigg[\sum_{v,w\in\mathbb{T}_n}\mathds1_{\{X_v=x,v \text{ is typical}\}}\mathds1_{\{X_w=x,w\text{ is typical}\}}\bigg]\\
&=\mathbb{E}\bigg[\sum_{k=0}^{n}\sum_{(v,w)\in\mathcal{N}_n^{(k)}}\mathds1_{\{X_v=x,v \text{ is typical}\}}\mathds1_{\{X_w=x,w\text{ is typical}\}}\bigg]\\
&=\sum_{k=0}^{n}\mathbb{E}\bigg[\sum_{(v,w)\in\mathcal{N}_n^{(k)}}\mathds1_{\{X_v=x,v \text{ is typical}\}}\mathds1_{\{X_w=x,w\text{ is typical}\}} \bigg].
\end{align*}
Computing the expectations for every $k\in\N$, by first conditioning on $\GWT$, gives
\begin{align*}
\mathbb{E}&\bigg[\sum_{(v,w)\in\mathcal{N}_n^{(k)}}\mathds{1}_{\{X_v=x,v \text{ is typical}\}}\mathds{1}_{\{X_w=x,w\text{ is typical}\}}\,\Big|\,\GWT\bigg] \\
& = \sum_{(v,w)\in\mathcal{N}_n^{(k)}}\Prob(X_v=X_w=x,\,v,w \text{ are typical}\,|\,\GWT).
\end{align*}
Note that if $v,w\in\mathbb{T}_n$ split at time $k$ at some vertex $y\in G$, the condition that both $v$ and $w$ are typical implies that $\mu^{*k}(y)\geq \e^{-k(\avez+\delta)-K}$, and they can only intersect at vertices $x\in G$ for which $\mu^{*n}(x)\leq \e^{-n(\avez-\delta)+K}$. Using the Chapman-Kolmogorov equation, we get that for such $x,y$ it holds
$$\mu^{*(n-k)}(y^{-1}x)\leq\frac{\mu^{*n}(x)}{\mu^{*k}(y)}\leq e^{-(n-k)\avez+(n+k)\delta+2K}.$$
We set 
$$\mathcal{A}_n^{(k)}=\{y\in G:\, \mu^{*k}(y)\geq \e^{-k(\avez+\delta)-K}\}\quad \text{and}\quad  \mathcal{C}_n = \{x\in G:\,\mu^{*n}(x)\leq \e^{-n(\avez-\delta)+K}\},$$
and we write $Z=X_{v_k}=X_{w_k}$. By summing over all possible locations $y\in \mathcal{A}_n^{(k)}$ where the particles split, i.e. summing over all possible values of $Z$, we obtain for every $x\in\mathcal{C}_n$ that
\begin{align*}
\Prob(X_v=X_w=x,\,v,w\text{ are typical}\,|\,\GWT) &= \sum_{y\in\mathcal{A}_n^{(k)}}\Prob(X_v = X_w = x,\,v,w\text{ are typical}\,|\,Z=y,\GWT)\Prob(Z=y\,|\,\GWT) \\
& \leq \sum_{y\in\mathcal{A}_n^{(k)}} \Prob(X_v = X_w = x\,|\,Z=y,\GWT)\Prob(Z=y\,|\,\GWT)\\
& = \sum_{y\in\mathcal{A}_n^{(k)}} \mu^{*k}(y)\mu^{*(n-k)}(y^{-1}x)^2\\
& \leq \Big(\sum_{{y\in\mathcal{A}_n^{(k)}}} \mu^{*k}(y)\mu^{*(n-k)}(y^{-1}x)\Big)\e^{-(n-k)\avez+(n+k)\delta+2K} \\
& \leq  \Big(\sum_{y\in G} \mu^{*k}(y)\mu^{*(n-k)}(y^{-1}x)\Big)\e^{-(n-k)\avez+(n+k)\delta+2K}\\
&=\mu^{*n}(x)\e^{-(n-k)\avez+(n+k)\delta+2K}.
\end{align*}
This further implies that
$$
\mathbb{E}\left[M_n^{\text{typ}}(x)^2\,|\,\GWT\right] \leq \mu^{*n}(x)\sum_{k=0}^n N_n^{(k)}\e^{-(n-k)\avez+(n+k)\delta+2K},
$$
and after taking the expectation and summing over all vertices in $\mathcal{C}_n$ we get
\begin{align*}
\E\bigg[\sum_{x\in G}M_n^{\text{typ}}(x)^2\bigg] = \E\bigg[\sum_{x\in \mathcal{C}_n}M_n^{\text{typ}}(x)^2\bigg] &\leq C\sum_{x\in \mathcal{C}_n}\mu^{*n}(x)\sum_{k=0}^n \rho^{2n-k}\e^{-(n-k)\avez+(n+k)\delta+2K}\\
& \leq C\rho^n\e^{2n\delta+2K}\sum_{k=0}^n \rho^{n-k}\e^{-(n-k)(\avez+\delta)}\\
& \leq C\rho^n\e^{2n\delta+2K}\sum_{j=0}^n \rho^{j}\e^{-j\avez},
\end{align*}
and this proves the claim.
\end{proof}

For a given $c > 0$, we define the set of good sites of $G$ in the small-$\rho$ regime as
$$
\mathcal{G}_n^{(1)} := \mathcal{G}_{n,\delta,K,c}^{(1)}:=\big\{x\in G:\,M_n^{\text{typ}}(x)\leq (1+c)^n\big\},
$$
and the set of badly behaved sites in the small-$\rho$ regime as
$$
\mathcal{B}_n^{(1)}:=\mathcal{B}_{n,\delta,K,c}^{(1)}:=G\setminus \mathcal{G}_{n,\delta,K,c}^{(1)}.
$$
In the large-$\rho$ regime, for any $\varepsilon>0$ we define the set of good sites by
$$
\mathcal{G}_n^{(2)}:=\mathcal{G}_{n,\delta,K,\varepsilon}^{(2)}:=\big\{x\in G:\, M_n^{\text{typ}}(x)\leq \rho^n\e^{-n(\avez-\varepsilon)}\big\},
$$
and the set of badly behaved sites by
$$
\mathcal{B}_n^{(2)}:=\mathcal{B}_{n,\delta,K,\varepsilon}^{(2)} := G\setminus \mathcal{G}_{n,\delta,K,\varepsilon}^{(2)}.
$$

\begin{theorem}\label{thm: main BRW}
Assuming \ref{ass: finite entropy}, \ref{ass: almost sure survival}, and \ref{ass: sec mom GW}, the following holds.
\begin{enumerate}[label = (\roman*)]
\item Assume $\rho\leq\e^{\avez}$. Then for every $c > 0$ there exists $\delta_0>0$ such that, for all $\delta\leq \delta_0$, it holds that
$$
\lim_{K\to\infty}\Prob\big(\lim_{n\to\infty}\empirical_n^{\text{typ}}\big(\mathcal{G}_{n,\delta,K,c}^{(1)}\big) = 1\big) = 1.
$$
\item Assume $\rho> \e^{\avez}$. Then for every $\varepsilon>0$ there exists $\delta_0>0$ such that, for all $\delta\leq \delta_0$, it holds that
$$
\lim_{K\to\infty}\Prob\big(\lim_{n\to\infty}\empirical_n^{\text{typ}}\big(\mathcal{G}_{n,\delta,K,\varepsilon}^{(2)}\big) = 1\big) = 1.
$$
\end{enumerate}
\end{theorem}

Before turning to the proof, we explain how this theorem implies Theorem \ref{thm: main 1}.
\begin{proof}[Proof of Theorem \ref{thm: main 1}]
\textit{(i)} Let $\rho\leq \e^{\avez}$. By Lemma \ref{lem: trivial upper}
$$
\limsup_{n\to\infty}\frac{H(\empirical_n)}{n}\leq \lim_{n\to\infty}\frac{\log Z_n}{n} = \log\rho.
$$
To derive the corresponding lower bound, fix an arbitrary $c > 0$. By Theorem \ref{thm: main BRW}\textit{(i)}, for  $\delta>0$ sufficently small,  it holds that $\lim_n\empirical_n^{\text{typ}}(\mathcal{G}_n^{(1)}) = 1$ with high probability as long as $K>0$ is sufficiently large. Lemma~\ref{lem: Lemma on entropy} yields
\begin{align}\label{eq:empiricalentropy}
\liminf_{n\to\infty}\frac{H(\empirical_n)}{n}\geq \liminf_{n\to \infty} \frac{\#\typ_n}{Z_n}H\big( \empirical_n^{\text{typ}}\big)
\end{align}
and $H\big( \empirical_n^{\text{typ}}\big)\ge -\sum_{x\in \mathcal{G}_n^{(1)}}\empirical_n^{\text{typ}}(x)\log \empirical_n^{\text{typ}} \ge -\empirical_n^{\text{typ}}\big(\mathcal{G}_n^{(1)}\big)\log \big(\frac{(1+c)^n}{\rho^n}\big)$ by the definition of good sites in the small-$\rho$ regime.
Lemma~\ref{lem: lemma on typical and odd particles} gives $\lim_{n\to \infty} Z_n^{-1}\#\typ_n=W^{-1}W_{\delta,K}^{\text{typ}}$ arbitrarily close to $1$ for large enough $K$. Therefore, when $\lim_{n\to\infty}\empirical_n^{\text{typ}}\big(\mathcal{G}_{n,\delta,K,c}^{(1)}\big) = 1$, which by Theorem~\ref{thm: main BRW}(i) is an event of arbitrarily close to $1$ probability as $K$ is chosen large enough,  it holds:
\begin{align*}
\liminf_{n\to\infty}\frac{H(\empirical_n)}{n}&\geq \liminf_{n\to\infty} \frac{\#\typ_n}{Z_n}\cdot\empirical_n^{\text{typ}}\big(\mathcal{G}_n^{(1)}\big)\cdot(\log\rho-\log(1+c))\\
&\geq W^{-1}W_{\delta,K}^{\text{typ}}(\log\rho-\log(1+c)).
\end{align*}
Letting $K$ go to infinity and using Lemma \ref{lem: lemma on typical and odd particles} finishes the proof as $c>0$ can be made arbitrarly small.

\textit{(ii)} Let $\rho> \e^{\avez}$.  For the upper bound, the definition of $\typ_n$ yields
$$
\e^{n(\avez+\delta)+K} =\sum_{x\in G}\e^{n(\avez+\delta)+K}\mu^{*n}(x) \geq \sum_{x\in\textsf{supp}(\empirical_n^{\text{typ}})}\e^{n(\avez+\delta)+K}\mu^{*n}(x)\geq \#\textsf{supp}(\empirical_n^{\text{typ}}),
$$
which together with Lemma \ref{lem: Lemma on entropy} implies
\begin{align*}
\limsup_{n\to\infty}\frac{H(\empirical_n)}{n} &\leq \limsup_{n\to\infty}\Big(\frac{\#\typ_n}{Z_n}\frac{H(\empirical_n^{\text{typ}})}{n}+ \frac{\#\odd_n}{Z_n}\frac{H(\empirical_n^{\text{odd}})}{n}+\frac{\log 2}{n}\Big)\\
&\leq \limsup_{n\to\infty}\Big(\frac{\#\typ_n}{Z_n}(\avez+\delta)+\frac{K}{n}+\frac{\#\odd_n}{Z_n}\frac{\log Z_n}{n}+\frac{\log2}{n}\Big)\\
&=W^{-1}W_{\delta,K}^{\text{typ}}(\avez+\delta) + W^{-1}W_{\delta,K}^{\text{odd}}\log\rho.
\end{align*}
By letting $K\to\infty$ and using Lemma \ref{lem: lemma on typical and odd particles}, we obtain the desired upper bound. For the lower bound fix $\varepsilon>0$. We use inequality~(\ref{eq:empiricalentropy}) with the lower bound
\[
H(\empirical_n^{\text{typ}}) \ge -\sum_{x \in \mathcal{G}_{n,\delta,K,\varepsilon}^{(2)}} \empirical_n^{\text{typ}}(x)  \log  \empirical_n^{\text{typ}}(x) \ge -\empirical_n^{\text{typ}}\big( \mathcal{G}_{n,\delta,K,\varepsilon}^{(2)}\big) \log \Big(\frac{\rho^n \e^{-n(\avez-\varepsilon)}}{\#\typ_n} \Big).
\]

 By Theorem \ref{thm: main BRW}, for $\delta>0$ sufficiently small  and  $K>0$ sufficiently large, with high probability it holds $\lim_n \empirical_n^{\text{typ}}(\mathcal{G}_n^{(2)})=1$. Thus
\begin{align*}
\liminf_{n\to\infty}\frac{H(\empirical_n)}{n}&\geq \liminf_{n\to\infty} \frac{\#\typ_n}{Z_n}\cdot\empirical_n^{\text{typ}}\big(\mathcal{G}_{n,\delta,K,\varepsilon}^{(2)}\big)\cdot\Big(\avez-\varepsilon+\frac{1}{n}\log\frac{\#\typ_n}{\rho^n}\Big)\\
& = W^{-1}W_{\delta,K}^{\text{typ}}(\avez-\varepsilon),
\end{align*}
and letting $K\to\infty$ and using Lemma \ref{lem: lemma on typical and odd particles} we also obtain a matching lower bound, and so the claim is proven.
\end{proof}

\subsection{Small-$\rho$ regime}
\begin{proof}[Proof of Theorem \ref{thm: main BRW}{(i)}]
Assume $\rho\leq\e^{\avez}$. Recall that, for $x\in G$, $M^{\text{typ}}_n(x)$ is the number of typical particles at  $x$. It suffices to prove that 
$\lim_n M_n^{\text{typ}}(\mathcal{B}_n^{(1)})/\rho^n = 0$
almost surely, since  it holds $\lim_{K}\Prob(W_{\delta,K}^{\text{typ}}=0)=0$ and $\lim_K W_{\delta,K}^{\text{typ}} = W$ almost surely by Lemma~\ref{lem: lemma on typical and odd particles}. We have
$$
M_n^{\text{typ}}\big(\mathcal{B}_n^{(1)}\big) = \sum_{x\in\mathcal{B}_n^{(1)}}M^{\text{typ}}_n(x)\leq \sum_{x\in \mathcal{B}_n^{(1)}}M_n^{\text{typ}}(x)\frac{M_n^{\text{typ}}(x)}{(1+c)^n}\leq \frac{1}{(1+c)^n}\sum_{x\in G}M_n^{\text{typ}}(x)^2,
$$
which together with $\rho\e^{-\avez}\leq 1$ and Lemma \ref{lem: two norm typical}, implies the existence of a constant $C>0$ such that
$$
\E\bigg[\frac{M_n^{\text{typ}}\big(\mathcal{B}_n^{(1)}\big)}{\rho^n}\bigg] \leq C\e^{2K}\frac{n\e^{2n\delta}}{(1+c)^n}.
$$
Choosing $\delta_0 < \log(1+c)/2$, and using the Borel-Cantelli Lemma completes the proof.
\end{proof}

\subsection{Large-$\rho$ regime}
\begin{proof}[Proof of Theorem \ref{thm: main BRW}{(ii)}]
Assume $\rho>\e^{\avez}$. In view of Lemma \ref{lem: lemma on typical and odd particles},
it suffices to prove that almost surely $\lim_n M_n^{\text{typ}}(\mathcal{B}_n^{(2)})/\rho^n = 0$. Using the same approach as in the small-$\rho$ regime, we obtain
$$
M_n^{\text{typ}}\big(\mathcal{B}_n^{(2)}\big) = \sum_{x\in\mathcal{B}_n^{(2)}}M^{\text{typ}}_n(x)\leq \sum_{x\in \mathcal{B}_n^{(2)}}M_n^{\text{typ}}(x)\frac{M_n^{\text{typ}}(x)}{\e^{-n(\avez-\varepsilon)}\rho^n}\leq \frac{1}{\e^{-n(\avez-\varepsilon)}\rho^n}\sum_{x\in G}M_n^{\text{typ}}(x)^2.
$$
This, together with Lemma \ref{lem: two norm typical}, implies that there exists a constant $C>0$, such that 
$$
\E\bigg[\frac{M_n^{\text{typ}}\big(\mathcal{B}_n^{(2)}\big)}{\rho^n}\bigg] \leq C\e^{2K}\frac{\e^{2\delta n}}{\e^{\varepsilon n}} \sum_{j=0}^n \big(\e^{\avez}\rho^{-1}\big)^j \le C\e^{2K}(n+1)\frac{\e^{2\delta n}}{\e^{\varepsilon n}},
$$
since $\rho^{-1}\e^{\avez}\leq 1$. Choosing $\delta_0 < \varepsilon/2$ and applying Borel-Cantelli lemma completes the proof. 
\end{proof}

\subsection*{Comments and open questions}

Our results naturally lead to  further interesting research questions on the asymptotic entropy and the empirical distributions of branching random walks, a few of which we highlight below.

\textbf{Second order estimates for $H(\empirical_n)$.}
Although this paper proves that the almost sure limit $H(\empirical_n)/n$ exists and identifies its value, it is natural to ask for finer asymptotics. What are the correct normalizations $f(n)$ and $g(n)$ such that, in the small-$\rho$ regime, $(H(\empirical_n)-n\log\rho)/f(n)$ and, in the large-$\rho$ regime, $(H(\empirical_n)-n\avez)/g(n)$ converge in distribution as $n\to\infty$ to a nontrivial limitting random variable? Can the limiting variable be identified?

\textbf{Phase transition in the limit measure.}
As already mentioned in the introduction, the phase transition from weak to strong survival can also be described from the view of the boundary behaviour of branching random walks; see \cite{dussaule-wang-yang-brw-hyper, LimSetHyperbolic} once again for details. 
The natural question at this point is if the phase transition in the entropy is also reflected in the boundary behaviour of the branching random walk. In this case the quantity of interest would not be the limit set, but the limit measure $\empirical_\infty=\lim_n\empirical_n$, which exists as a weak limit in view of \cite{WoessKaim} and \cite{Candellero-Hutchcroft-BRW-2023}, and it is a random measure supported on the boundary of the underlying state space. The case of branching Brownian motion on hyperbolic spaces has been investigated in \cite{Geldbach-2026-BBM}.

\textbf{Question.}
Under the assumptions of Theorem \ref{thm: main 1}, is the limit empirical measure $\empirical_\infty$ singular with respect to the exit measure of the random walk $(G,\mu)$ when $\rho \le e^{\avez}$, and absolutely continuous with respect to that exit measure when $\rho > e^{\avez}$? We have strong evidence that this should be indeed the case for $\rho<e^{\avez}$. 

\textbf{Clustering of typical particles.}
Another natural question is whether Theorem \ref{thm: main BRW} and Theorem \ref{thm: main ind empirical} can be strengthened. More precisely: 
\begin{itemize}
    \item When $\rho \le e^{\avez}$, can we replace the “small exponential” bound in Theorem \ref{thm: main ind empirical}(i) showing that only polynomially many typical particles meet for both independent random walks and for branching random walks?
    \item When $\rho > e^{\avez}$, does the empirical distribution of the typical particles approximate the step distribution of the underlying random walk arbitrarily well, that is, can  Theorem \ref{thm: main BRW}(ii) be strengthened to Theorem \ref{thm: main ind empirical}(ii)?
\end{itemize}

\textbf{Asymptotic Rényi-entropies.} The $\alpha$-Rényi-entropy of a measure $\nu$ is defined as
$$\entropy(\alpha,\nu):=\frac{1}{1-\alpha}\log\sum_{g\in G}\nu(g)^\alpha,$$
and it extends the Shannon entropy in that taking the limit as $\alpha\to 1$ yields the usual Shannon entropy.
In \cite{MR4771974} it is shown that the asymptotic Renyi-entropies exist, i.e. for every $\alpha\in[0,\infty)$ the quantity $h(\alpha,\mu)=\lim_n H(\alpha,\mu^{*n})/n$ exists and the function $\alpha\mapsto h(\alpha,\mu)$ is continuous on $[0,\infty)\setminus\{1\}$ and decreasing. This paper deals with the case $\alpha=1$.  
We conjecture a similar phase transitions for all  $\alpha\geq 1$: if $\rho\leq \e^{h(\alpha,\mu)}$, then $\lim_n H(\alpha,\empirical_n)/n = \log\rho$ almost surely, while if $\rho>\e^{h(\alpha,\mu)}$ then $\lim_n H(\alpha,\mu^{*n})/n = h(\alpha,\mu)$ almost surely. For $\alpha = \infty$, the Rényi entropy equals the negative logarithm of the spectral radius, i.e., $h(\infty,\mu)=-\log r$, so the conjecture would imply that in the weak survival phase only subexponentially many particles coalesce, while in the strong survival phase the maximal number of particles coalescing is of order $r^n\rho^n$.
For $\alpha \in [0,1)$, we expect that, beyond the small-$\rho$ and large--$\rho$  regimes, there is also an intermediate regime governed by the large deviations rate function of the underlying random walk. However, we do not yet have sufficient evidence to formulate a conjecture.

\paragraph{Acknowledgments and funding information.}
	Thanks to Hanna Oppelmayer and Nadia Fellin for discussions at an early stage of this project.
	This research was funded in part by the Austrian Science Fund (FWF) [10.55776/\allowbreak{}PAT3123425]. The research of J. Brieussel was funded in part by the ANR-22-CE40-0004 GoFR, the ANR-24-CE40-3137
PLAGE and the JSPS Invitational Fellowship L25508. 
    
\bibliographystyle{alpha}
\bibliography{lit}

\texttt{Jeremie Brieussel}, University of Montpellier, France.
\texttt{jeremie.brieussel@umontpellier.fr}

\texttt{Robin Kaiser}, Technische Universität München, Germany.
\texttt{ro.kaiser@tum.de}

\texttt{Martin Klötzer}, Universität Innsbruck, Austria. 
\texttt{Martin.Kloetzer@uibk.ac.at}

\texttt{Ecaterina Sava-Huss}, Universität Innsbruck, Austria.
\texttt{Ecaterina.Sava-Huss@uibk.ac.at}

\end{document}